\documentclass[a4paper,12pt]{article}

\usepackage[pdftex]{graphicx}
\usepackage[section] {placeins}
\usepackage{enumerate}

\usepackage{latexsym}
\usepackage{amsfonts}
\usepackage{amsmath}
\usepackage{amsthm}

\newtheorem{thm}{Theorem}[section]
\newtheorem{exa}[thm]{Example}
\newtheorem{lem}[thm]{Lemma}
\newtheorem{cor}[thm]{Corollary}

\newtheorem{conj}[thm]{Conjecture}

\newtheorem{ques}{Question}

\newcommand{\Z}{\mathbb{Z}}

\newcommand{\spl}{\mathop{\rm SL}\nolimits}

\usepackage{fullpage}

\begin{document}

\title{Sequences in Dihedral  Groups with \\ Distinct Partial Products }

\author{M.~A.~Ollis\footnote{Email address: \texttt{matt@marlboro.edu}} \\ 
            Marlboro College, P.~O.~Box A, Marlboro, \\ Vermont 05344, USA. }

\date{}
             
\maketitle

\begin{abstract}
Given a subset~$S$ of the non-identity elements of the dihedral group of order~$2m$, is it possible to order the elements of~$S$ so that the partial products are distinct?   This is equivalent to the sequenceability of the group when~$|S| = 2m-1$ and so it is known that the answer is yes in this case if and only if~$m>4$.
We show that the answer is yes when $|S| \leq 9$ and~$m$ is an odd prime other than~3, when~$|S| = 2m-2$ and~$m$ is even or prime, and when~$|S| = 2m-2$ for many instances of the problem when~$m$ is odd and composite. We also consider the problem in the more general setting of arbitrary non-abelian groups and discuss connections between this work and the concept of strong sequenceability.


\end{abstract}

\section{Introduction}

Let~$G$ be a multiplicatively-written group with identity element~$e$ and let~${\bf g} = (g_1, g_2, \ldots g_k)$ be a sequence of elements of~$G \setminus \{ e \}$.  Define the {\em partial product sequence} of~${\bf g}$ to be~${\bf h} = (h_0, h_1, h_2, \ldots, h_k)$ where~$h_0 = e$ and~$h_i = g_1g_2 \cdots g_i$ for~$1 \leq i \leq k$. 

The following conjecture is investigated in~\cite{CMPP18}:

\begin{conj}\label{conj:abhard}
Let~$G$ be an abelian group and let~$S \subseteq G \setminus \{ e\}$ such that the product of all of the elements in~$S$ is not the identity.  Then there exists an ordering of the elements of~$S$ such that the elements in its partial product sequence are distinct.
\end{conj}

This conjecture generalises an earlier one of Alspach that was limited to cyclic~$G$.  Conjecture~\ref{conj:abhard} is known to be true in the following cases:
\begin{itemize}
\item when~$|S| \leq 9$ \cite{AL,BH05},
\item when~$|S| \leq 10$ and $G$ is cyclic of prime order \cite{HOS19},
\item when~$|S| = |G|-1$ \cite{Gordon61},
\item when~$|S| = |G|-2$ and $G$ is cyclic \cite{BH05},
\item when~$|S| = |G|-3$ and~$G$ is cyclic of prime order~\cite{HOS19},
\item when~$|G| \leq 21$ \cite{BH05, CMPP18}.
\end{itemize}

In \cite{CMPP18} the question of dropping the requirement that~$G$ be abelian is raised and quickly rejected upon consideration of a counterexample.  The question of which subsets of which non-abelian groups do satisfy the conditions remains, and that is the question we study here.

\begin{ques}\label{ques:nonabhard}
Let~$G$ be a finite group and let~$S$ be a subset of~$G \setminus \{ e \}$ of size~$k$ such that there is some ordering of the elements of~$S$ whose product is not the identity.  Is there an ordering of the elements of~$S$ such that the elements in its partial product sequence are distinct?
\end{ques}

In order to begin to address this, we introduce some notation and terminology.

Let~$G$ be a group of order~$n$.
As in the first paragraph, let~${\bf g} = (g_1, g_2, \ldots, g_k)$ be an arrangement of elements of~$G \setminus \{e\}$ with partial product sequence~${\bf h} = (h_0, h_1, h_2, \ldots, h_k)$.  Suppose the elements of~${\bf g}$ are distinct and let~$S = \{ g_1, g_2, \ldots, g_k \}$.

If the elements of~${\bf h}$ are all distinct then~${\bf h}$ is a {\em basic directed $S$-terrace} for~$G$ and~${\bf g}$ is the associated~{\em $S$-sequencing} of~$G$.  In the case~$k = n-1$ (and so~$S = G \setminus \{ e \}$), ${\bf h}$ is a {\em basic directed terrace} for~$G$ and~${\bf g}$ is the associated~{\em sequencing} of~$G$.   A group with a sequencing is called~{\em sequenceable}.

The study of sequencings in non-abelian groups originated in~\cite{Gordon61} and is surveyed in~\cite{OllisSurvey}.  Note that it is always possible to order the non-identity elements of a non-abelian group to give a non-identity product. 

The three non-abelian groups of orders 6 and~8 are not sequenceable~\cite{Gordon61}, hence the answer to Question~\ref{ques:nonabhard} in the cases when~$S$ contains all of the non-identity elements of such a group is no~\cite{CMPP18}.  {\em Keedwell's Conjecture} is that all other non-abelian groups are sequenceable; that is, that the answer to Question~\ref{ques:nonabhard} is yes when~$S$ contains all of the non-identity elements of a non-abelian group of order at least~10.

In the next section, when cataloguing possible structures of sets for small~$k$,  we see more instances of sets~$S$ for which the answer to Question~\ref{ques:nonabhard} is no.
In Section~\ref{sec:polymeth} we use the Non-Vanishing Corollary to Alon's Combinatorial Nullenstatz to show that the answer is always yes in dihedral groups of order~$2m$ when~$m>3$ is prime and~$|S| \leq 9$.  In Sections~\ref{sec:largek} and~\ref{sec:gr} we show that the answer is yes in dihedral groups of order~$2m$ for~$k \geq 2m-2$ when~$m>4$ is even or prime, and for many instances of the question for composite~$m$.

Alspach and Kalinowski, see~\cite{AL}, have asked a closely related question regarding ``strong sequenceability" of groups.  The main difference to Question~\ref{ques:nonabhard} is that the product of all the elements of a successful ordering is permitted to be the identity (in abelian groups one has no control over this value).  We give the necessary definitions and consider the implications of our work for the strong sequenceability question in Section~\ref{sec:ss}.

\section{Small~$k$, general groups}\label{sec:smallk}

For related conjectures that are more limited in their claims or restricted to abelian groups a case-based approach has been used prove them for small values of~$k$, including Conjecture~\ref{conj:abhard} for~$k \leq 5$~\cite{AL,ADMS,CMPP18}.  In this section we start this process for Question~\ref{ques:nonabhard}.

\begin{thm}\label{th:small_nonabhard}
The answer to Question~\ref{ques:nonabhard} is yes for~$k \leq 4$, with the following two exceptions:
\begin{itemize} 
\item $|S|=4$, with $S = \{ x,x^{-1},y,z\}$ and $xyz  = x^{-1}zy = e$,
\item $|S|=4$, with $S = \{ w,x,y,z \}$ and $wxy=wyz=wzx=xzy=e$.
\end{itemize}
\end{thm}

\begin{proof}
We consider potential $S$-sequencings and break into cases depending on what might cause a given sequence not to be an $S$-sequencing.  Those causes will be subsequences of elements whose product is the identity.  An added wrinkle compared to the abelian case is that it is not immediately obvious when a set of words which are the identity imply a contradiction.  We use the group theory software package GAP~\cite{GAP} to determine when a set of such words implies that either the subgroup generated by the elements in question is trivial (which will imply a contradiction) or abelian (in which case the problem is reduced to one already solved in~\cite{ADMS}).  We frequently use that if $sx = e$ for some string of group elements $s$ then also $xs = e$.

When~$k \in \{ 1,2 \}$ the result is immediate.  For~$k=3$ let~$S = \{ g_1, g_2, g_3 \} \subseteq G \setminus \{ e \}$ and suppose that $g_1g_2g_3 \neq e$.  If $(g_1, g_2, g_3)$ is not an $S$-sequencing then it must be that either~$g_2 = g_1^{-1}$ or $g_2 = g_3^{-1}$.  In the former case $(g_1, g_3, g_2)$ is an $S$-sequencing, in the latter case $(g_2, g_1, g_3)$ is.

For the~$k=4$ case, we consider~$S$ of three forms: $\{x,x^{-1},y,y^{-1}\}$, $\{x, x^{-1}, y, z \}$ and $\{w,x,y,z \}$ where any elements that are inverses are so indicated.
  
If~$S$ is of the first form then without loss of generality a sequence whose product is not the identity is $(x,y,x^{-1},y^{-1})$. This is an $S$-sequencing as the product of none of the two- or three-element subsequences is the identity.

Next consider~$S$ of the form~$\{x, x^{-1}, y, z \}$.  Conjugates of $y$ or $z$ by $x$ cannot be the identity. We know there is one ordering whose product is not the identity.  If this is not an $S$-sequencing, then there must be a sequence of length~3 of elements of $x^{\pm 1}$, $y$ and $z$ whose product is the identity.   Without loss of generality, assume $xyz=e$.  Now consider the ordering $(x,y,x^{-1},z)$.  The product of all four cannot be the identity as this implies $x^{-1}=e$.  The only way that this cannot be an $S$-sequencing is if $yx^{-1}z = e$, which gives the first exception in the statement of the theorem.

Finally, consider $S = \{ w,x,y,z\}$.  Assume without loss of generality that $wxy=e$.  If $(y,z,x,w)$ is not an $S$-sequencing then either $zxw=e$ or $yzxw=e$.  In the former case, considering $(x,w,y,z)$ and then $(w,x,z,y)$ leads to either a successful $S$-sequencing or the second  exception in the statement of the theorem.  In the latter case, $(z,w,y,x)$ fails to be an $S$-sequencing only if $zwyx=e$, in which case $(w,z,x,y)$ is an $S$-sequencing.
\end{proof}

Both exceptions in Theorem~\ref{th:small_nonabhard} are necessary. To see this for the first, let
$$D_{2m} = \langle u,v : u^m = e = v^2, vu = u^{m-1}v \rangle$$
be the dihedral group of order~$2m$.
Consider 
$S =  \{ u, u^2, v, u^2v \} \subseteq D_6$. It is straightforward to check there is no $S$-sequencing.
 
For the second, let $\spl(2,3)$ be the special linear group of $2 \times 2$ matrices with determinant~1 over the field with three elements. Then there is no $S$-sequencing for
$$S = \left\{ \begin{pmatrix}  
                   0 & 1 \\
                   2 & 1
                   \end{pmatrix}  ,
                   \begin{pmatrix}  
                   1 & 1 \\
                   2 & 0 
                   \end{pmatrix}  ,
                   \begin{pmatrix}  
                   2 & 0 \\
                   2 & 2 
                   \end{pmatrix}  ,
                   \begin{pmatrix}  
                   2 & 1 \\
                   0 & 2 
                   \end{pmatrix}     \right\}$$
in~$\spl(2,3)$.

These two exceptions, and the three with sizes~5 and~7 that come from the non-sequenceable groups of orders~6 and~8, are not the only instances of sets~$S$ that do not admit an $S$-sequencing.  For example, in~$D_8$ the set
$$S = \{ u^2, v,uv,u^2v,u^3v \}$$
of size 5 does not have an $S$-sequencing.  

As a further example, let $Q_8$ be the quaternion group of order~8 and let~$z$ be its unique involution.  Then~$S = Q_8 \setminus \{ e,z \}$ is a set of size~6 with no $S$-sequencing.

\section{Small~$k$, dihedral groups}\label{sec:polymeth}

Before moving to the main method of the section, we give a general construction that works in all dihedral groups and obviates the need for the most computationally intensive case in Theorem~\ref{th:polymeth}, the main result of this section.

Let~$C_m = \langle u \rangle$, a normal cyclic subgroup of~$D_{2m}$ of order~$m$, and let~$C_m v$ be its coset.

\begin{lem}\label{lem:r0}
If $S \subseteq C_m v \subseteq D_{2m}$, then $D_{2m}$ has an $S$-sequencing.
\end{lem}

\begin{proof}
Let $S = \{ u^{a_i}v : 1 \leq i \leq k \}$ with $a_1 > \cdots > a_k$.  Then the sequence $(u^{a_1}v,  \ldots, u^{a_k}v)$ is an $S$-sequencing with partial products
$$( e, u^{a_1}v, u^{a_1 - a_2}, u^{a_1-a_2+a_3}v, u^{a_1-a_2+a_3-a_4}, \ldots ).$$
The ordering of the~$a_i$ guarantees that these elements are distinct.
\end{proof}

We now pursue a new approach that applies specifically to dihedral groups of order twice an odd prime.   It uses Alon's Non-Vanishing Corollary and we are able answer Question~\ref{ques:nonabhard} for these groups up to~$k = 9$.  Perhaps more importantly, it embodies an approach that could plausibly be extended to all values of~$k$ in these groups.  This generalises a method developed for cyclic groups in~\cite{HOS19}.

The Non-Vanishing Corollary was introduced in~\cite{Alon99}; for a short direct proof see~\cite{Michalek10}.

\begin{thm}\label{th:pm}{\rm (Non-Vanishing Corollary)}
Let ~$F$ be an arbitrary field, and let $f = f(x_1, x_2, \ldots, x_k)$ be a polynomial in~$F[x_1, x_2, \ldots, x_k]$.  Suppose the degree~$deg(f)$ of~$f$ is $\sum_{i=1}^k \gamma_i$, where each~$\gamma_i$ is a nonnegative integer, and suppose the coefficient of~$\prod_{i=1}^r x_i^{\gamma_i}$ in~$f$ is nonzero.  Then if~$A_1, A_2, \ldots, A_k$ are subsets of~$F$ with~$|A_i| > \gamma_i$, there are $a_1 \in A_1, \ldots, a_k \in A_k$ so that~$f(a_1, a_2, \ldots, a_k) \neq 0$.
\end{thm}

For~$m$ prime, we will take~$F$ to be~$\Z_m$, the integers modulo~$m$ considered as a field.   We need  a polynomial that is non-zero if and only if we feed it a solution to the problem.  Then, if we can find a monomial with the required properties in relation to our set-up, the Non-Vanishing Corollary will give a positive answer to that instance of Question~\ref{ques:nonabhard}.

Suppose our set~$S \subseteq D_{2m} \setminus \{ e \}$ has~$r$ elements in~$C_m$ and~$s$ elements in~$C_m v$.  We look for a solution of a particular form, which varies slightly with the parity of~$s$.  Before considering the general case, we look at a small example.

\begin{exa}\label{ex:h32}
Consider the case when~$|S|=5$  with three elements in~$C_m \setminus \{ e \}$ and two in $C_m v$.  There are various ways in which we might arrange such elements to look for a sequence with distinct partial products.  

One potential form for a successful sequence is
$$( u^{x_1}, u^{y_1}v, u^{x_2}, u^{x_3}, u^{y_2}v ),$$
which has partial products
$$( e, u^{x_1}, u^{x_1+y_1}v, u^{x_1 + y_1 - x_2}v, u^{x_1 + y_1 - x_2 - x_3}v, u^{x_1 + y_1 - x_2 - x_3 - y_2} ).$$

Consider the polynomial
$$(x_2-x_1)(x_3-x_1)(x_3-x_2)(y_2-y_1)(x_1+y_1-x_2-x_3-y_2)(y_1-x_2-x_3-y_2)(-x_2-x_3)$$
(which will be called~$\pi_{3,2}$ in the general method described in the remainder of this section).  This polynomial is nonzero if and only if our sequence has the elements it must have and has distinct partial products.

To apply the Non-Vanishing Corollary we need a monomial that divides $x_1^2x_2^2x_3^2y_1y_2$ with a non-zero coefficient.  The polynomial has degree~7, so this is plausible to ask for.  Indeed, there is one: $x_1^2x_2x_3^2y_1y_2$, which has coefficient~$6$.

Hence for prime~$m>3$ (to which~$6$ is coprime) every subset~$S$ of size~$5$ of~$D_{2m}$ that has  three elements in~$C_m \setminus \{ e \}$ and two in $C_m v$ has an~$S$-sequencing.
\end{exa}

For the general case, we need to arrange the $r$ elements of $C_m$ and $s$ elements of $C_m v$.   There are many potentially-successful ways to do this.  For a given form, we can get some crude information about how likely it is that it is possible to assign the elements so that the partial products are distinct based on our  knowledge of which coset each of the partial products is in.  

Looking at that list of partial products, a hard rule is that we cannot have more than $m$ of the products in~$C_m$ or  in~$C_m v$.  Two softer rules that guided the choice are to a) try to make the number of elements in the partial product in each coset roughly equal to each other (which tends to lower the degree of the polynomial under consideration compared to other options), and to b) mimic patterns that have been successful in finding sequencings for dihedral groups (that is, the case $r=m-1$ and $s=m$).

We first consider the case when~$s$ is odd.  In this case use the form
$$(u^{x_1}, u^{x_2} ,\ldots, u^{x_{\lfloor r/2 \rfloor} }, u^{y_1}v, u^{y_2}v, \ldots, u^{y_{s}}v, u^{x_{\lfloor r/2 \rfloor + 1}}, u^{x_{\lfloor r/2 \rfloor + 2}}, \ldots, u^{x_{r}}).$$
The sequence of partial products is
\begin{flushleft}
$( u^{z_0}, u^{z_1}, \ldots, u^{z_{\lfloor r/2 \rfloor} }, u^{t_1}v,  u^{z_{\lfloor r/2 \rfloor+1} }, u^{t_2}v, \ldots, u^{z_{\lfloor r/2 \rfloor+(s-1)/2} },$
\end{flushleft}
\begin{flushright}
$  u^{t_{(s+1)/2 } }v, u^{t_{(s+1)/2 + 1} }v , \ldots, u^{t_{(s+1)/2 + \lceil r/2 \rceil  }   }v  ). $
\end{flushright}

Each~$z_i$ and~$t_i$ is a linear combination of the~$x_i$ and~$y_i$ with all coefficients~$\pm 1$.   Let~$r = 2p+ \delta$, for~$\delta \in \{ 0,1 \}$, and~$s = 2q+1$.  Then
$$
z_i=
\begin{cases}
0  \mbox{\ \  if } i=0   \\
x_1 + x_2 + \cdots x_i \mbox{\ \  if }1 \leq i \leq p \\
z_{p} + y_1 - y_2 + y_3 - y_4 + \cdots + y_{2(i-p) - 1} - y_{2(i-p)} \mbox{\ \  if }p+1 \leq i \leq p+q
\end{cases}
$$
and
$$
t_i=
\begin{cases}
z_p + y_1 -  y_2 + y_3 - y_4 + y_5 + \cdots - y_{2(i-1)} + y_{2i-1} \mbox{\ \  if } i \leq q+1 \\
t_{q+1} - x_{p+1} - x_{p+2} - \cdots - x_{i - q + p -1}      \mbox{\ \  if } q+2 \leq i \leq p+q+1+\delta.   
\end{cases}
$$

The polynomial
$$\prod_{1 \leq i < j \leq 2p} (x_j - x_i)  \prod_{1 \leq i < j \leq 2q+1}  (y_j - y_i)
       \prod_{0 \leq i < j \leq p+q} (z_j - z_i)  \prod_{1 \leq i < j \leq p+q+1} (t_j - t_i).$$ 
is not~0 if and only if the assignment to the variables~$x_1, \ldots, x_r, y_1, \ldots, y_s$ solves our problem.  The Non-Vanishing Corollary is generally easier to apply when the polynomial in question has a lower degree.  To this end, we look to remove some redundant factors.  

We know that~$e \not\in S$, and the~$(y_{i+1} - y_{i})$ terms guarantee that there can be no adjacent repeated involutions in the sequence.  This implies that~$z_{i+1} \neq z_i$ and~$t_{i+1} \neq t_i$ for each~$i$.  So define~$\pi_{r,s}$ for~$r = 2p+ \delta$, for~$\delta \in \{ 0,1 \}$, and~$s = 2q+1$ by:
$$\pi_{r,s} = \prod_{1 \leq i < j \leq 2p+\delta} (x_j - x_i)  \prod_{1 \leq i < j \leq 2q+1}  (y_j - y_i)
       \prod_{\substack{0 \leq i < j \leq p+q \\ j \neq i+1 }} (z_j - z_i)  \prod_{\substack{1 \leq i < j \leq p+q+1+\delta \\ j \neq i+1}} (t_j - t_i).$$ 
This is the polynomial to which we shall apply the Non-Vanishing Corollary.

If~$s=0$ then the problem reduces to one in the cyclic group~$C_m$, and is addressed in~\cite{HOS19}.
If~$s \geq 2$ is even we look for a solution of the form
$$(u^{x_1}, u^{x_2} ,\ldots, u^{x_{\lfloor r/2 \rfloor} }, u^{y_1}v, u^{y_2}v, \ldots, u^{y_{s-1}}v, u^{x_{\lfloor r/2 \rfloor + 1}}, u^{x_{\lfloor r/2 \rfloor + 2}}, \ldots, u^{x_{r}}, u^{y_{s}}v  ).$$
That is, we adjoin the additional element of~$C_mv$ to the end of the form used for odd~$s$.
The sequence of partial products is
\begin{flushleft}
$( u^{z_0}, u^{z_1}, \ldots, u^{z_{\lfloor r/2 \rfloor} }, u^{t_1}v,  u^{z_{\lfloor r/2 \rfloor+1} }, u^{t_2}v, \ldots, u^{z_{\lfloor r/2 \rfloor+(s-2)/2} },$
\end{flushleft}
\begin{flushright}
$  u^{t_{s/2 } }v, u^{t_{s/2 + 1} }v , \ldots, u^{t_{s/2 + \lceil r/2 \rceil  }   }v, u^{z_{\lfloor r/2 \rfloor+s/2 } } ).  $
\end{flushright} 

Let~$r = 2p+\delta$, with~$\delta \in \{ 0, 1 \}$ as before, and set~$s = 2q+2$.
Then~$z_i$ and~$t_i$ have the same values as before with the addition that~$ z_{p+q+1} =  t_{p+q+1 + \delta} - y_{2q+2} $.
If~$r \neq 0$, let~$\pi_{r,s}$ be 
$$(z_{p+q+1} - z_{p+q}) \prod_{1 \leq i < j \leq 2p + \delta} (x_j - x_i)  \prod_{1 \leq i < j \leq 2q+2}  (y_j - y_i)
       \prod_{\substack{0 \leq i < j \leq p+q+1 \\ j \neq i+1 }} (z_j - z_i)  \prod_{\substack{1 \leq i < j \leq p+q+1 + \delta \\ j \neq i+1}} (t_j - t_i)   $$ 
Again $\pi_{r,s}(x_1, \ldots, x_r, y_1, \ldots, y_s)  \neq 0$ if and only if $x_1, \ldots, x_r, y_1, \ldots, y_s$ gives a solution to the problem.   

For completeness, note that when~$r=0$ and~$s$ is even we have $z_{p+q+1} - z_{p+q} = y_{s-1} - y_s$ and hence we can omit the factor $(z_{p+q+1} - z_{p+q})$ in this case as the negative of this factor is included in the second product.  However, this case is covered by Lemma~\ref{lem:r0} anyway.

As~$\pi_{r,s}$ is homogeneous, any monomial where the exponent on each~$x_i$ is less than~$r$ and the exponent on each~$y_j$ is less than~$s$ is suitable for use in the Non-Vanishing Corollary.

\begin{thm}\label{th:polymeth}
Let~$m>3$ be an odd prime.  If~$S \subseteq D_{2m} \setminus \{ e \}$ with~$|S| \leq 9$ then $D_{2m}$ has an~$S$-sequencing.
\end{thm}

\begin{proof}

If~$|S| \leq 4$ then the result follows from Theorem~\ref{th:small_nonabhard}.
For larger values of~$|S|$, let~$r$ be the number of elements of~$S$ that are in~$C_m$ and~$s=|S|-r$ the number that are in~$C_m v$. 
If~$r=0$ then the result follows from Lemma~\ref{lem:r0}.
If~$r = |S|$ then, as noted earlier, the result follows from the cyclic group version of the conjecture which is proved for~$|S| \leq 10$  for prime~$m$ in~\cite{HOS19}.    

For each~$r$ with~$1 \leq r \leq |S|-1$ we apply the method of Example~\ref{ex:h32}.  
Tables~\ref{tab:k5}--\ref{tab:k9} collect the pertinent information concerning the polynomial~$\pi_{r,s}$ and one of its monomials for $|S|$ from~5 through to~9.  
\end{proof}

\begin{table}
\caption{Some details of~$\pi_{r,s}$ for~$|S|=5$}\label{tab:k5}
$$
\begin{array}{rrrrr}
\hline
r & \deg(\pi) & {\rm monomial} & {\rm coefficient} & {\rm prime \ factors} \\
\hline
1 & 9 & y_1^3y_3^3y_4^3 & 4  & 2 \\
2 & 6 & x_2y_1y_2^2y_3^2 & -3  & 3 \\
3 & 7 & x_1^2x_2x_3^2y_1y_2 & 6 & 2,3 \\ 
4 & 8 & x_2^3x_3^2x_4^3 & 1 & - \\
\hline
\end{array}
$$
\end{table}

\begin{table}
\caption{Some details of~$\pi_{r,s}$ for~$|S|=6$}\label{tab:k6}
$$
\begin{array}{rrrrr}
\hline
r & \deg(\pi) & {\rm monomial} & {\rm coefficient}  & {\rm prime \ factors} \\
\hline
1 & 14 & y_2^2y_3^4y_4^4y_5^4 & -4 & 2 \\
2 & 12 & x_2y_1^3y_2^2y_3^3y_4^3 & 16 & 2 \\
3 & 10 & x_1^2x_3^2y_1^2y_2^2y_3^2 & -3 & 3 \\
4 & 12 & x_1^3x_2^3x_3^2x_4^3y_2 & -4 & 2 \\
5 & 14 & x_1^2x_2^4x_4^4x_5^4 & -2 & 2 \\
\hline
\end{array}
$$
\end{table}

\begin{table}
\caption{Some details of~$\pi_{r,s}$ for~$|S|=7$}\label{tab:k7}
$$
\begin{array}{rrrrr}
\hline
r & \deg(\pi) & {\rm monomial} & {\rm coefficient}  & {\rm prime \ factors} \\
\hline
1 & 22 & y_2^5y_3^5y_4^5y_5^5y_6^2 & -16 & 2 \\
2 & 17 & x_2y_2^4y_3^4y_4^4y_5^4 & -4 & 2 \\
3 & 16 & x_1^2x_3^2y_1^3y_2^3y_3^3y_4^3 & 32 & 2 \\
4 & 16 & x_1^3x_2^3x_3^3x_4^3y_2y_3^2 & 2 & 2 \\
5 & 18 & x_1^4x_2^4x_4^4x_5^4y_1y_2 & 12 & 2,3\\
6 & 21 & x_2^5x_3^5x_4x_5^5x_6^5 & -2 & 2  \\
\hline
\end{array}
$$
\end{table}

\begin{table}
\caption{Some details of~$\pi_{r,s}$ for~$|S|=8$}\label{tab:k8}
$$
\begin{array}{rrrrr}
\hline
r & \deg(\pi) & {\rm monomial} & {\rm coefficient}  & {\rm prime \ factors} \\
\hline
1 & 30 & y_2y_3^6y_4^5y_5^6y_6^6y_7^6 & -64 & 2 \\
2 & 26 & x_2y_2^5y_3^5y_4^5y_5^5y_6^5 & -72 & 2,3  \\
3 & 22 &   x_1^2x_2^2x_3^2y_1^4y_3^4y_4^4y_5^4 & -1 & - \\
4 & 23 &    x_1^2x_2^2x_3^3x_4^3y_1^3y_2^3y_3^3y_4^3 & -48 & 2,3 \\
5 & 24 & x_1^4x_2^4x_3^3x_4^4x_5^4y_2y_3^2  & -3 & 3 \\
6 & 26 & x_1^4x_2^5x_3^5x_4x_5^5x_6^5y_2 & 48 & 2,3 \\
7 & 30 & x_2^6x_3x_4^5x_5^6x_6^6x_7^6 & 1 & - \\
\hline
\end{array}
$$
\end{table}

\begin{table}
\caption{Some details of~$\pi_{r,s}$ for~$|S|=9$}\label{tab:k9}
$$
\begin{array}{rrrrr}
\hline
r & \deg(\pi) & {\rm monomial} & {\rm coefficient}  & {\rm prime \ factors} \\
\hline
1 &  41  &  y_2y_3^7y_4^7y_5^7y_6^7y_7^6y_8^6      & 720  & 2,3,5 \\
2 &  34 &   x_2y_2^6y_3^6y_4^6y_5^6y_6^5y_7^4      & -512 & 2 \\
3 & 31 &  x_1x_2^2x_3^2y_1^5y_2^5y_3^5y_4y_5^5y_6^5     & -384  & 2,3 \\
4 &  28  & x_1^3x_2^3x_3^3x_4^3y_1^4y_2y_3^4y_4^3y_5^4 & 12 & 2,3 \\
5 &  29 & x_1^3x_2^4x_3^3x_4^4x_5^4y_1^3y_2^3y_3^3y_4^2 & 8 & 2\\
6 & 30 & x_1x_2^5x_3^5x_4^5x_5^5x_6^5y_1^2y_3^2 & -16 & 2 \\
7 & 35 & x_1^5x_2^6x_3^6x_5^5x_6^6x_7^6y_2 & 64 & 2 \\
8 & 40 & x_2^7x_3x_4^7x_5^4x_6^7x_7^7x_8^7 & -3 & 3 \\
\hline
\end{array}
$$
\end{table}

Note that in the proof of Theorem~\ref{th:polymeth} when~$|S|=5$ and~$r=2$ the coefficient is not coprime to~3 (and neither is the coefficient on any other viable monomial).  This is necessarily the case as~$D_6$ is not sequenceable.




In all cases in the proof of Theorem~\ref{th:polymeth} there were many monomials with non-zero coefficients.  The ones in the tables were chosen to have only small prime factors.  Other monomials could have been used in combination, provided that their greatest common divisor has only small prime factors.  A more general theoretical approach to solving the problem by finding monomial coefficients might take advantage of this.

Thus the answer to Question~\ref{ques:nonabhard} is yes for~$|S| \leq 9$ in dihedral groups of order twice a prime, with the exceptions noted for~$D_6$ and $|S| \in \{4,5\}$ in the previous section.  For~$D_{10}$ this answers the question completely.  A deeper understanding of~$\pi_{r,s}$ is a conceivable route to removing (or weakening) this condition on~$S$.

\section{Large~$k$}\label{sec:largek}

In this section we consider subsets of $D_{2m} \setminus \{ e \}$ of size at least~$2m-2$. An affirmative answer to Question~\ref{ques:nonabhard} for~$k = 2m-1$ and $m \geq 5$ follows immediately from known constructions:

\begin{thm}\label{th:dihed_n-1}{\rm \cite{Isbell90, Li97}}
The dihedral group of order~$n$ is sequenceable if and only if~$n \geq 10$. 
\end{thm}

Existing results also get us some of the way for the~$2m-2$ case, via the general result given in Lemma~\ref{lem:rseq}.  In order to state that we need a notion closely related to sequenceability.

Let ${\bf a} = (a_1, a_2, \ldots, a_{n-1})$ be a cyclic arrangement of the non-identity elements of~$G$ (i.e.~$a_{n-1}$ is considered to be adjacent to $a_1$) and define~${\bf b} = (b_1, b_2, \ldots, b_{n-1})$ by $b_i = b_i^{-1}b_{i+1}$ where the indices are considered modulo~$(n-1)$ (so $b_{n-1} = b_{n-1}^{-1}b_1$).     If the elements of~${\bf b}$ are distinct then ${\bf a}$ is a {\em directed  rotational terrace} for~$\Z_n$ and~${\bf b}$ is its associated {\em  rotational sequencing}.  Clearly, the directed  rotational terrace determines the  rotational sequencing; the reverse is also true.

(Note: there are several different but equivalent definitions in the literature; see, for example,~\cite{AAAP11,FGM,Keedwell83,OW15}.  We avoid the more common, but less descriptive, names ``R-sequencing" and ``directed R-terrace" for these concepts to bypass confusion with our $S$-sequencings and directed $S$-terraces.)

\begin{lem}\label{lem:rseq}
Let $G$ be a group of order~$n$ and let $S \subseteq G \setminus \{ e \}$ with $|S| = n-2$.
If $G$ has a rotational sequencing then it also has an $S$-sequencing.
\end{lem}

\begin{proof}
Let~$S = G \setminus \{ e, x \}$ for some~$x \in G$.   Let ${\bf b} = (b_1, b_2, \ldots, b_{n-1})$ be a rotational sequencing of~$G$ indexed so that~$b_{n-1} = x$.  The partial products of  $(b_1, b_2, \ldots, b_{n-2})$ are all distinct because otherwise there would be a repeat in the directed rotational terrace associated with~${\bf b}$.  Hence  $(b_1, b_2, \ldots, b_{n-2})$ is the required $S$-sequencing.
\end{proof}

Bode and Harborth use this method (but not this terminology) to prove Conjecture~\ref{conj:abhard} for the~$k = n-2$ case for cyclic groups of odd order.  More generally, the recent proof~\cite{AKP17} of a conjecture of Friedlander, Gordon and Miller~\cite{FGM} similarly implies that Conjecture~\ref{conj:abhard} holds for the~$k = n-2$ case for all abelian groups that do not have a single involution.

For dihedral groups, this approach covers the cases where the group has order a multiple of~4:
\begin{thm}
Let $m$ be even and $S \subseteq D_{2m} \setminus \{ e \}$ with~$|S| = 2m-2$.  Then~$D_{2m}$  has an $S$-sequencing.
\end{thm}

\noindent
Proof. The dihedral group~$D_{2m}$ has a rotational sequencing if and only if~$m$ is even~\cite{Keedwell83}.  Apply Lemma~\ref{lem:rseq}.
\qed

We also take our cue from Bode and Harborth's methodology when~$m$ is odd.  Their proof that Conjecture~\ref{conj:abhard} holds when~$k = n-2$ for cyclic groups of even order implicitly uses the following result:

\begin{lem}\label{lem:seq}
If~$G$ has a sequencing with first element~$x$ then $G$ has an~$S$-sequencing for~$S = G \setminus \{ e,x \}$.
\end{lem}

\begin{proof}
If the sequencing is~$(x,  b_2, \ldots, b_{n-1})$ then~$(b_2, \ldots, b_{n-1})$ must be an~$S$-sequencing else we would have a repeat somewhere in the directed terrace associated with the sequencing.
\end{proof}

Therefore, for odd~$m$, our task becomes to a construct sequencing for~$D_{2m}$ with first element~$x$,  for each possible choice of~$x$.  Noting that there is an automorphism that maps one element to another in~$D_{2m}$ ($m$ still odd) if and only if the two elements have the same order reduces the problem to finding sequencings that have first elements of all possible orders.   In order to follow this path, we introduce and generalise some of the constructions of Isbell~\cite{Isbell90} for sequencings of dihedral groups.

To begin, we say a bit more about how sequencings function and introduce graceful permutations.

Let~$G$ be a group of order~$n$ and let~${\bf g} = (g_1, g_2, \ldots g_{n-1})$ be a sequencing with basic directed terrace~${\bf h} = (h_0, h_1, h_2, \ldots, h_{n-1})$.   Then any sequence~${\bf h'} = (h'_0, h'_1, h'_2, \ldots, h'_{n-1})$ such that~${h'}_{i-1}^{-1}h'_i = g_i$ for each~$i$ is called a {\em directed terrace} for~$G$.  The basic directed terrace~${\bf h}$ is a directed terrace and a sequence is a directed terrace if and only if it is of the form $x{\bf h} = (xh_0, xh_1, xh_2, \ldots, xh_{n-1})$.  Not requiring that directed terraces be basic removes an unnecessary restriction when attempting to build a sequencing via a directed terrace.

Given a directed terrace ${\bf h} = (h_0, h_1, h_2, \ldots, h_{n-1})$ for~$G$ there are two simple ways to obtain further directed terraces~\cite{Bailey84}.  First, we may reverse~${\bf h}$ to give the directed terrace $(h_{n-1}, h_{n-2}, \ldots, h_0)$.   Second we may find the unique value of~$i$ such that $h_{n-1}^{-1}h_0 = g_i$ and produce the directed terrace $(h_{i+1}, h_{i+2}, \ldots, h_{n-1}, h_0, \ldots, h_{i})$.  Call this the {\em translation} of~${\bf h}$.  Thus from each directed terrace we can produce three new directed terraces: its reverse, its translation and the reverse of its translation.

Let~${\bf a} = (a_1, a_2, \ldots, a_n)$ be an arrangement of the integers~$\{ 0  ,1, \ldots, n-1 \}$.  If the {\em sequence of absolute differences} ${\bf b} = (b_1, b_2, \ldots, b_{n-1})$ defined by~$b_i = |a_{i+1} - a_i|$ consists of the integers~$\{ 1, 2, \ldots, n-1 \}$ then~${\bf a}$ is a {\em graceful permutation}.  A graceful permutation is equivalent to a graceful labelling of a path with~$n$ vertices; see~\cite{GallianSurvey} for more details about graceful labellings.

We shall need graceful permutations with various properties in our constructions and will investigate them further in the next section.
For now, the following example has the constructions we use to prove the conjecture when~$m$ is prime.

\begin{exa}\label{ex:gp1}
The sequence $(0,2\ell-1, 1, 2\ell-2, 2, 2\ell-3, \ldots, \ell-1, \ell)$ is a graceful permutation of length~$2\ell$ 
and $(0, 2\ell, 1, 2\ell-1, 2, 2\ell-2, \ldots, \ell, \ell+2, \ell+1)$ is a graceful permutation of length~$2\ell+1$.  These are known as the {\em Walecki Constructions}~\cite{Alspach08}.

When~$\ell$ is odd, the sequence 
\begin{flushleft}
$(\ell, 0, 2\ell, 1, 2\ell-1, \ldots, (\ell-1)/2, (3\ell+1)/2, $
\end{flushleft} 
\begin{flushright}
$(3\ell-1)/2, (\ell+1)/2, (3\ell-3)/2, (\ell+3)/2, \ldots, \ell+1, \ell-1)$
\end{flushright}
is a graceful permutation of length~$2\ell+1$ \cite{Isbell90}.  The similar construction
$$(\ell, 2\ell-1, 0, 2\ell-2, 1, \ldots, (3\ell-1)/2, (\ell-1)/2, (\ell+1)/2, (3\ell-3)/2, (\ell+3)/2,  \ldots, \ell+1, \ell-1 )$$
is easily checked to be a graceful permutation of length~$2\ell$. 
\end{exa}

Isbell~\cite{Isbell90} gives three constructions for sequencings of dihedral groups~$D_{2m}$ where~$m$ is odd.  In that paper the concern is to get one sequencing for each order.  We require more, so the following descriptions work with arbitrary sequences of integers that have the properties on which Isbell relied (graceful permutations in the first two, a slight generalisation thereof in the third) in place of the specific sequences used by Isbell.  Further,~\cite{Isbell90} only covers the~$m \equiv 3 \pmod{4}$ cases for the second and third constructions; in addition to these we give slight variations that include the $m \equiv 1 \pmod{4}$ cases (although we do not have the integer sequences required to make use of the third construction one). 

\vspace{2mm}
\noindent
{\bf Isbell's first construction.}
Let~$m = 4\ell+1$.  Let~$(a_1, a_2, \ldots, a_{2\ell})$ be a graceful permutation of length~$2\ell$ with differences~$(b_1, b_2, \ldots, b_{2\ell-1})$ (not absolute differences; here~$b_i = a_{i+1} - a_i$) and such that~$a_{2\ell}  = \ell$.  Consider the symbols modulo~$4\ell+1$ rather than as integers, then the sequence
\begin{flushleft}
$( u^{b_1}, u^{b_2}, \ldots, u^{b_{2\ell-1}}; u^{2\ell}; v, uv, u^2v, \ldots, u^{2\ell-1}v; u^{4\ell}v;$
\end{flushleft}
\begin{flushright} 
$    u^{2\ell}v, u^{2\ell+1}v, \ldots, u^{4\ell-1}v; u^{2\ell+1}; u^{-b_{2\ell-1}}, u^{-b_{2\ell-2}}, \ldots, u^{-b_1}  )$
\end{flushright}
is a sequencing for~$D_{8\ell+2}$ (where semi-colons are used to help indicate the pattern).  The associated directed terrace is
\begin{flushleft}
$( u^{a_1}, u^{a_2}, \ldots, u^{a_{2\ell}}; u^{3\ell} ;  u^{3\ell}v, u^{3\ell-1}, u^{3\ell+1}v, u^{3\ell-2}, \ldots, u^{4\ell-1}v, u^{2\ell};     $
\end{flushleft}
\begin{flushright}
$  u^{2\ell-1}v, u^{4\ell}, u^{2\ell}v, u^{4\ell-1}v, \ldots, u^{3\ell+1}, u^{3\ell-1}v;  
u^{2\ell-2-a_{2\ell}}v, u^{2\ell-2-a_{2\ell-1}}v, 
\ldots, u^{2\ell-2 - a_{1}}v   ).$
\end{flushright}

\vspace{2mm}
\noindent
{\bf Isbell's second construction.}
Let~$m=4\ell+3$ with~$\ell$ odd.  Let~$(a_1, a_2, \ldots, a_{2\ell+1})$ be a graceful permutation of length~$2\ell+1$ with differences~$(b_1, b_2, \ldots, b_{2\ell})$ (again, not absolute differences) and such that~$a_1 = \ell$ and~$a_{2\ell+1} = \ell-1$.  Consider the symbols modulo~$4\ell+3$ rather than as integers, then the sequence
\begin{flushleft}
$( u^{b_1}, u^{b_2}, \ldots, u^{b_{2\ell}}; u^{2\ell+2}; v, uv, u^2v, \ldots, u^{2\ell}v; u^{4\ell+2}v;$
\end{flushleft}
\begin{flushright} 
$    u^{2\ell+1}v, u^{2\ell+2}v, \ldots, u^{4\ell+1}v; u^{2\ell+1}; u^{-b_{1}}, u^{-b_{2}}, \ldots, u^{-b_{2\ell} }  )$
\end{flushright}
is a sequencing for~$D_{8\ell+6}$.  The associated directed terrace is
\begin{flushleft}
$( u^{a_1}, u^{a_2}, \ldots, u^{a_{2\ell+1}}; u^{3\ell+1} ;  u^{3\ell+1}v, u^{3\ell}, u^{3\ell+2}v, u^{3\ell-1}, \ldots, u^{2\ell+1}  ,u^{4\ell+1}v;     $
\end{flushleft}
\begin{flushright}
$  u^{4\ell+2}, u^{2\ell}v, u^{4\ell+1},  u^{2\ell+1}v,  \ldots, u^{3\ell+2}, u^{3\ell}v; u^{a_{1} - 1}v, u^{a_{2} - 1}v, 
\ldots, u^{a_{2\ell+1}-1}v   ).$
\end{flushright}
Although Isbell did not consider this case, essentially the same construction works for $m=4\ell+1$ with~$\ell$ odd.  Let~$(a_1, a_2, \ldots, a_{2\ell})$ be a graceful permutation of length~$2\ell$ with differences~$(b_1, b_2, \ldots, b_{2\ell-1})$ (not absolute differences) and such that~$a_1 = \ell$ and~$a_{2\ell} = \ell-1$.  Consider the symbols modulo~$4\ell+1$ rather than as integers, then the sequence
\begin{flushleft}
$( u^{b_1}, u^{b_2}, \ldots, u^{b_{2\ell-1}}; u^{2\ell+1}; v, uv, u^2v, \ldots, u^{2\ell-1}v; u^{4\ell}v;$
\end{flushleft}
\begin{flushright} 
$    u^{2\ell}v, u^{2\ell+1}v, \ldots, u^{4\ell-1}v; u^{2\ell}; u^{-b_{1}}, u^{-b_{2}}, \ldots, u^{-b_{2\ell-1} }  )$
\end{flushright}
is a sequencing for~$D_{8\ell+2}$.  The associated directed terrace is
\begin{flushleft}
$( u^{a_1}, u^{a_2}, \ldots, u^{a_{2\ell}}; u^{3\ell} ;  u^{3\ell}v, u^{3\ell-1}, u^{3\ell+1}v, u^{3\ell-2}, \ldots, u^{4\ell-1}v  ,u^{2\ell};     $
\end{flushleft}
\begin{flushright}
$  u^{2\ell-1}v, u^{4\ell}, u^{2\ell}v, u^{4\ell+1},   \ldots, u^{3\ell+1}, u^{3\ell-1}v; u^{a_{1} - 1}v, u^{a_{2} - 1}v, 
\ldots, u^{a_{2\ell}-1}v   ).$
\end{flushright}
The only reason that this construction does not work for even~$\ell$ is that the required graceful permuation cannot exist~\cite{Gvozdjak04}; we shall see more about this in the next section.

\vspace{2mm}

For the third construction we need a new concept, closely related to those of $\hat{\rho}$-labellings (a.k.a. nearly graceful labellings) and holey $\alpha$-labellings described in~\cite[Section~3.3]{GallianSurvey}.     Let~${\bf a} = (a_1, a_2, \ldots, a_n)$ be an arrangement of~$n$ of the integers~$\{ 0  ,1, \ldots, n \}$.  If the {\em sequence of absolute differences} ${\bf b} = (b_1, b_2, \ldots, b_{n-1})$ defined by~$b_i = |a_{i+1} - a_i|$ consists of the integers~$\{ 1, 2, \ldots, n-1 \}$ then we call~${\bf a}$ a {\em cracked graceful permutation}.  The missing element of~$\{ 0  ,1, \ldots, n \}$ is called the {\em crack}.

\begin{exa}\label{ex:hgp}
For even values of~$\ell$ we give the cracked permutation of length~$2\ell - 1$ with crack~$\ell - 2$ given by Isbell~\cite{Isbell90}.  
For~$\ell = 2,4,8$ the sequences $(3,1,2)$,  $(5,0,6,7,3,1,4)$ and
$$(9,5,7,12,11,2,10,4,14,3,0,13,1,15,8)$$
respectively have the required properties.

For other~$\ell$ the first part of the construction varies as~$\ell$ varies modulo~$6$, always having the form of an ad hoc sequence of elements followed by ``zigzag" sequences of length~$6$ in a regular pattern.  

For~$\ell \equiv 0 \pmod{6}$ it starts 
$$ \ell+1, \ell-3, \ell-1; \ell+4, \ell-6, \ell+4, \ell-5, \ell+2, \ell-4; \ell+7, \ell-9, \ell+6, \ell-8, \ell+5, \ell-7; \ldots$$
For~$\ell \equiv 2 \pmod{6}$ with~$\ell \geq 14$ it starts
\begin{flushleft}
$\ell+1, \ell-6, \ell+3, \ell-1, \ell-3, \ell+2, \ell-4, \ell+4, \ell-7, \ell+5, \ell -5; $
\end{flushleft}
\begin{flushright}
$\ell+8, \ell-10, \ell+7, \ell-9, \ell+6, \ell-8; \ell+11, \ell-13, \ell+10, \ell-12, \ell+9, \ell-11; \ldots$
\end{flushright}
For~$\ell \equiv 4 \pmod{6}$ with~$\ell \geq 10$ it starts
\begin{flushleft}
$ \ell+1, \ell-1, \ell-5, \ell+3, \ell-4, \ell+2, \ell-3;$
\end{flushleft}
\begin{flushright}
$ \ell+6, \ell-8, \ell+5, \ell-7, \ell+4, \ell-6; \ell+9, \ell-11, \ell+8, \ell-10, \ell+7, \ell-9; \ldots$
\end{flushright}
In each case the last zigzag sequence of length~$6$ is
$$(3\ell-4)/2, \ell/2, (3\ell-6)/2, (\ell+2)/2, (3\ell-8)/2, (\ell+4)/2$$
and the cracked graceful permutation concludes with~$(l-2)/2$ followed by a long zigzag and two final ad hoc elements:
$$(3\ell-2)/2, (\ell-4)/2, 3\ell/2, (\ell-6)/2, \ldots, 0, 2\ell - 2; 2\ell-1, \ell.$$
\end{exa}

\noindent
{\bf Isbell's third construction.}  Let~$m=4\ell+1$ with~$\ell$ even. Let~$(a_1, a_2, \ldots, a_{2\ell-2})$ be a cracked graceful permutation of length~$2\ell-2$ with crack~$\ell-4$  and differences~$(b_1, b_2, \ldots, b_{2\ell-3})$ (not absolute differences) and such that~$a_1 = \ell-1$ and~$a_{2\ell-2} = \ell-2$.  Consider the symbols modulo~$4\ell+1$ rather than as integers, then the sequence
\begin{flushleft}
$( u^{b_1}, u^{b_2}, \ldots, u^{b_{2\ell-3}}; u^{2\ell+1}, u^{2\ell-2},u^{2\ell+2} ; v, uv, u^2v, \ldots, u^{2\ell-2}v;u^{4\ell-2}v, u^{4\ell-1}, u^{4\ell}v;$
\end{flushleft}
\begin{flushright} 
$    u^{2\ell-1}v, u^{2\ell}v, \ldots, u^{4\ell-3}v; u^{2\ell-1}; u^{-b_{1}}, u^{-b_{2}}, \ldots, u^{-b_{2\ell-3} }; u^{2\ell}, u^{2\ell+3}  )$
\end{flushright}
is a sequencing for~$D_{8\ell+2}$.  The associated directed terrace is
\begin{flushleft}
$( u^{a_1}, u^{a_2}, \ldots, u^{a_{2l-2}}; u^{3l-1}, u^{l-4}, u^{3l-2} ;  u^{3l-2}v, u^{3l-3}, u^{3l-1}v, u^{3l-4}, \ldots, u^{2l-1}, u^{4l-3}v;     $
\end{flushleft}
\begin{center}
$  u^{4l},  u^{4l-2}v, u^{4l-1};  u^{2l-3}v, u^{4l-2}, u^{2l-2}v,  u^{4l-3},  \ldots, u^{3l}, u^{3l-4}v;$ 
\end{center}
\begin{flushright}
$u^{a_{1} - 2}v, u^{a_{2} - 2}v, \ldots, u^{a_{2l-2}-2}v; u^{3l-3}v, u^{l-6}v   ).$
\end{flushright}
Let~$m=4\ell+3$ with~$\ell$ even. Let~$(a_1, a_2, \ldots, a_{2\ell-1})$ be a cracked graceful permutation of length~$2\ell-1$ with crack~$\ell-2$  and differences~$(b_1, b_2, \ldots, b_{2\ell-2})$ (not absolute differences) and such that~$a_1 = \ell+1$ and~$a_{2\ell-1} = \ell$.  Consider the symbols modulo~$4\ell+3$ rather than as integers, then the sequence
\begin{flushleft}
$( u^{b_1}, u^{b_2}, \ldots, u^{b_{2\ell-2}}; u^{2\ell+2}, u^{2\ell-1},u^{2\ell+3} ; v, uv, u^2v, \ldots, u^{2\ell-1}v;u^{4\ell}v, u^{4\ell+1}, u^{4\ell+2}v;$
\end{flushleft}
\begin{flushright} 
$    u^{2\ell}v, u^{2\ell+1}v, \ldots, u^{4\ell-1}v; u^{2\ell}; u^{-b_{1}}, u^{-b_{2}}, \ldots, u^{-b_{2\ell-2} }; u^{2\ell+1}, u^{2\ell+4}  )$
\end{flushright}
is a sequencing for~$D_{8\ell+6}$.  The associated directed terrace is
\begin{flushleft}
$( u^{a_1}, u^{a_2}, \ldots, u^{a_{2\ell-1}}; u^{3\ell+2}, u^{\ell-2}, u^{3\ell+1} ;  u^{3\ell+1}v, u^{3\ell}, u^{3\ell+2}v, u^{3\ell-1}, \ldots, u^{4\ell}v  ,u^{2\ell+1};     $
\end{flushleft}
\begin{center}
$  u^{2\ell-2}v,  u^{2\ell}, u^{2\ell-1}v;  u^{4\ell+2}, u^{2\ell}v, u^{4\ell+1},  u^{2\ell+1}v,  \ldots, u^{3\ell+3}, u^{3\ell-1}v;$ 
\end{center}
\begin{flushright}
$u^{a_{1} - 2}v, u^{a_{2} - 2}v, 
\ldots, u^{a_{2\ell-1}-2}v; u^{3\ell}v, u^{\ell-4}v   ).$
\end{flushright}
Similarly to the second construction, the parity restrictions on~$\ell$ are because otherwise the required cracked graceful permutations do not exist.  We prove this in the next section.

\begin{thm}\label{th:prime_m}
Let~$m$ be an odd prime.  If $S \subseteq D_{2m} \setminus \{ e \}$  with $|S|=2m-2$, then $D_{2m}$ has an $S$-sequencing.
\end{thm}

\begin{proof}
As~$m$ is prime each pair of elements of~$C_m \setminus \{ e \}$ are equivalent by automorphisms, as are each pair of elements of~$C_m v$.   By Lemma~\ref{lem:seq} it is therefore sufficient to find a sequencing with first element in~$C_m \setminus \{ e \}$ and a sequencing with first element in~$C_m v$. 

 Let~${\bf g} = (g_1, g_2 , \ldots g_{2m-1} )$  be the sequencing for~$D_{2m}$ constructed from Isbell's first construction when $m \equiv 1 \pmod{4}$, Isbell's second construction when $m \equiv 7 \pmod{8}$ and Isbell's third construction when $m \equiv 3 \pmod{8}$, in each case using the (cracked) graceful permuations given in the examples. Let~${\bf h} = (h_1, h_2 , \ldots h_{2m} )$ be the associated  directed terrace.

In all cases $g_1 = u^{b_1}$, where~$b_1$ is the first element of the differences of the (cracked) graceful permutation.  We have~$g_1 \in C_m \setminus \{ e \}$.

Now consider the translation of~${\bf h}$.  
When $m=4\ell+1$, the first element of the associated sequencing of the reverse of the translation is~$u^{2\ell-3}v$.
When~$m = 4\ell+3$ for even~$\ell$, the first element of its associated sequencing is $u^{4\ell-1}v$.    When~$m=4\ell+3$ for odd~$\ell$, the first element of the associated sequencing of the reverse of the translation is~$u^{4\ell}v$.  Each is in~$C_m v$. 
\end{proof}

To move to composite values of~$m$ we need to vary the first element in the sequencings we construct.  We do this by varying the (cracked) graceful permutations used in Isbell's constructions.  Note that as all elements of~$C_m v$ are equivalent under automorphisms when~$m$ is odd, the proof of Theorem~\ref{th:prime_m} gives sequencings that start with any such element for arbitrary odd~$m$.  We can therefore focus on sequencings starting with elements of~$C_m \setminus \{ e\}$.

\begin{lem}\label{lem:g2s}
Let~$(a_1, a_2, \ldots, a_{2\ell})$ be a graceful permutation of length~$2\ell$ with sequence of absolute differences~$(b_1, b_2, \ldots, b_{2\ell-1})$. Let~$(c_1, c_2, \ldots, c_{2\ell+1})$ be a graceful permutation of length~$2\ell+1$ with sequence of absolute differences~$(d_1, d_2, \ldots, d_{2\ell})$.  Then
\begin{enumerate}
\item if $a_{2\ell} = \ell$, then $D_{8\ell+2}$ has a sequencing with first element~$u^{b_1}$.
\item if $a_1= \ell$ and $a_{2\ell} = \ell-1$, then~$D_{8\ell+2}$ has a sequencing with first element~$u^{b_1}$ and a sequencing with first element~$u^{b_{2\ell-1}}$.
\item if $c_1 = \ell$ and $c_{2\ell+1} = \ell-1$, then~$D_{8\ell+6}$ has a sequencing with first element~$u^{d_1}$ and a sequencing with first element~$u^{d_{2\ell}}$.
\end{enumerate}
\end{lem}  

\begin{proof}
Part~1 follows from Isbell's first construction.  Parts~2 add~3 follow from Isbell's second, the first of each clause directly and the second after taking the reverse.  The distinction between absolute differences (here) and differences (in the Isbell constructions) is rendered moot as they give elements of~$D_{2m}$ that have the same order and hence are equivalent under automorphisms (as~$m$ is odd).
\end{proof}

\begin{lem}\label{lem:cg2s}
Let~$(a_1, a_2, \ldots, a_{2\ell-2})$ be a cracked graceful permutation of length~$2\ell-2$ with sequence of absolute differences~$(b_1, b_2, \ldots, b_{2\ell-2})$ and crack~$\ell-4$. Let~$(c_1, c_2, \ldots, c_{2\ell-1})$ be a cracked graceful permutation of length~$2\ell-1$ with sequence of absolute differences $(d_1, d_2, \ldots, d_{2\ell-2})$ with crack~$\ell-2$.  Then
\begin{enumerate}
\item if $a_1 = \ell-1$ and $a_{2\ell-2} = \ell-2$, then $D_{8\ell+2}$ has a sequencing with first element~$u^{b_1}$ and a sequencing with first element $u^{2\ell+3}$. 
\item if $c_1 = \ell+1$ and $c_{2\ell-1} = \ell$, then $D_{8\ell+6}$ has a sequencing with first element~$u^{d_1}$ and a sequencing with first element $u^{2\ell+4}$. 
\end{enumerate}
\end{lem}

\begin{proof}
These follow from Isbell's third construction and its reverse.
\end{proof}

In the next section we investigate the existence of (cracked) graceful permutations for use with these lemmas.

\section{Graceful Permutations}\label{sec:gr}

Our task in this section is to construct (cracked) graceful permuations that meet the criteria for use in one of Isbell's constructions and have endpoints and first/last differences that give a variety of orders for the first element of the dihedral group sequencing. 

Given a graceful permutation~${\bf a} = (a_1, a_2, \ldots, a_n)$ on the symbols~$\{0, \ldots, n-1 \}$ define the {\em complement} to be~${\bf \bar{a}} = (n-1-a_1, n-1-a_2, \ldots, n-1-a_n)$.  This is also a graceful permutation. 

The most flexible result from the previous section is the first clause of Lemma~\ref{lem:g2s}, which requires a graceful permutation with specified first difference and last element and gives results for~$D_{8\ell+2}$.  We work with this first.

It is known that for a graceful permutation of length~$n$, a first difference of~$d$, for $1 \leq d \leq n-1$, is possible except when~$d=2$ and $n \in \{ 4,5,8 \}$~\cite{HOS19}.  Useful tools in the proof of that result are ``twizzler terraces" and these can also help us here. To make the most of these constructions, we need to be able to control the elements at each end of a graceful permutation.  
Necessary conditions are known:

\begin{thm}\label{th:gvoz}{\rm \cite{Gvozdjak04}}
Let~$0 \leq x,y \leq n-1$ and~$x \neq y$.  If there is a graceful permutation of length~$n$ with first element~$x$ and last element~$y$ then
\begin{itemize}
\item  $|x-y|$ has the same parity as $\lfloor n/2 \rfloor$
\item $|x-y| \leq n/2$
\item $(n-1)/2 \leq x+y \leq (3n-3)/2$
\end{itemize}
\end{thm}

Gv\"ozdjak conjectures that these conditions are also sufficient \cite{Gvozdjak04}.  Call this {\em Gv\"ozdjak's Conjecture}.   It is also known that for any~$x < n$ there is a graceful sequence of length~$n$ that starts with~$x$, see any of~\cite{Cattell07, FFG83, Gvozdjak04}.

Suppose we write~$n = pq+r$, where $p,q \geq 1$ and~$r \geq p/2$.  It is possible to construct an {\em imperfect $p$-twizzler terrace}~${\bf g}$, which is a graceful permutation and here we shall call a {\em $p$-twizzler permuation}, as follows.  (Proof of the construction's correctness can be found in~\cite{OW11}.)  

Start with the Walecki Construction of length~$n$.  Divide the first~$pq$ elements into $q$ subsequences of length~$p$.  Reverse each subsequence while keeping the order of the subsequences intact.  Rearrange the final~$r$ elements into a translate of a graceful permutation of length~$r$ such that the absolute difference between the $pq$th and $(pq+1)$th elements of the sequence is~$r$ (the condition that~$r \geq p/2$ and the fact that there is a graceful permutation of length~$r$ that starts with any element are the crucial pieces that guarantee this is possible). 

\begin{exa}\label{ex:twizzler}
Consider $n=2\ell=22$ with $p=4$, $q=3$ and $r=10$.  The following is a $p$-twizzler permutation with these parameters where the last 10 elements are a translate of $(8,2,6,1,9,0,7,4,3,5)$:
$$(20,1,21,0;18,3,19,2;16,5,17,4;14,8,12,7,15,6,13,10,9,11)$$
(semicolons separate the subsequences).
\end{exa}

\begin{thm}\label{th:gvoz2constr}
Assume Gv\"ozdjak's Conjecture holds.  Then there is a $S$-sequencing for any~$S \subseteq D_{8\ell+2} \setminus \{e\}$ with~$|S| = 8\ell$.
\end{thm}

\begin{proof}
It is sufficient to consider the cases where~$4\ell + 1$ is composite and the non-identity element missing from~$S$ has order~$o$ with~$3 \leq o < 4\ell+1$.   Let~$d \in [l+1, 2l-1]$ such that $d$ has order~$o$ when considered modulo~$4\ell+1$.  Such a choice is possible when~$o=3$ as $(4\ell+1)/3 \in [l+1, 2l-1]$.  It is possible for larger~$o$ because in this case we must have~$\ell \geq 5$ and an element of every order other than~3 must appear in any list of~$\ell -1$ elements.

We construct a $p$-twizzler permutation of length~$2\ell = pq+r$ with~$p = 2\ell-d+1$ and~$q=1$.  It shall have first difference~$d$ and final element in~$\{ \ell-1, \ell \}$ and the result then follows by Lemma~\ref{lem:g2s}, possibly after taking the complement.

Suppose~$p$ is even.
If  Gv\"ozdjak's Conjecture holds, then there is a graceful permutation~${\bf g}$ of length~$r = 2\ell - p - 1$ with first element $2\ell -3p/2-1$ and final element in $\{ \ell-1-p/2, \ell - p/2 \}$.  Add~$p/2$ to each element to get a translate with first element~$2\ell-p-1$ and last element in~$\{ \ell-1, \ell \}$.  As $|(2\ell - p - 1) - 0| = r$, we have the required properties.

The odd case is similar.
\end{proof}

While the ability to vary~$q$ in the twizzler construction was not used in this proof, we can use this flexibility to exploit that it is known Gv\"ozdjak's Conjecture holds for~$n \leq 20$~\cite{Gvozdjak04}:

\begin{thm}\label{th:constr20}
Suppose that~$2\ell$ can be written $2\ell=pq+r$ with~$q\geq 1$ and~$p/2 \leq r \leq 20$. Then there is a $S$-sequencing for any~$S \subseteq D_{8\ell+2} \setminus \{e\}$, where~$|S| = 8\ell$ and the missing non-identity element of~$S$ has the same order as~$u^{2\ell-p+1}$.
\end{thm}

\begin{proof}
The method is exactly as in the proof of Theorem~\ref{th:gvoz2constr}, except we do not insist that~$q=1$.  In this more general setting, we need the graceful permutation of length~$r=2\ell-pq$ to have final element in $\{ \ell - 1-pq/2, \ell - pq/2 \}$ and an appropriate first element.  As Gv\"ozdjak's Conjecture holds for~$r$, we have complete choice of possible first element.
\end{proof}

Theorem~\ref{th:constr20} implies $D_{8\ell+2}$ has an $S$-sequencing for all $S \subseteq D_{8\ell+2} \setminus \{e\}$ with~$|S|=8\ell$ for many values of~$\ell$, including all~$\ell < 35$ for which~$4\ell+1$ is composite and
$$\ell \in \{ 36, 40, 42,  46, 51, 52,  54,  55, 63,  72, 75,82, 85, 90,  94 \}.$$
It gives partial results for many more.  The smallest value of~$\ell$ for which $4 \ell + 1$ is composite and Theorem~\ref{th:constr20} adds no new $S$-sequencings is~420.

When considering Isbell's second and third constructions, we need (cracked) graceful permutations with both first and last element specified.  Further, the first and last elements are both near the center of the possible values. Twizzler permutations, and other similar extension constructions in the literature, tend to push at least one of the values of the endpoints closer to an extreme and so are not very helpful for the current situation.   We are able to make some partial progress by introducing a new extension construction for graceful permutations.

Let~${\bf c} = (c_1, \ldots, c_{2p})$ be a graceful permutation of even length~$2p$ on $[0, 2p-1]$.  Say that~${\bf c}$ is {\em bipartite} if the odd-index elements $\{c_1, c_3, \ldots, c_{2p-1}\}$ are either all at most~$p-1$ or all at least~$p$.  That is, a bipartite graceful permutation has alternating ``small" and ``large" elements.  

In the more general theory of graceful labelings of graphs, bipartite graceful labelings are often known as {\em $\alpha$-labelings}, see~\cite{GallianSurvey}.   

\begin{lem}\label{lem:bgp}{\rm \cite{Kotzig73} }
For any~$x \in [0, p-1]$ there is a bipartite graceful permutation $(c_1, \ldots, c_{2p})$ with~$c_1 = x$.  
If $(c_1, \ldots, c_{2p})$ is a bipartite graceful permutation of length~$2p$ with $c_1 < p$, then~$c_{2p} = c_1 + p$.  
\end{lem}

We can now give the main construction.

\begin{thm}\label{th:insert}
Let ${\bf a} = (a_1, \ldots, a_q)$ be a graceful permutation of length~$q$.    Suppose ${\bf a}$ has adjacent elements~$x$ and~$y$ with $x < y < 2(y-x) = p$.  Then there is a graceful permutation of length~$2p+q$ with first element~$a_1 + p$ and last element~$a_q + p$.
\end{thm}

\begin{proof}
Let ${\bf c} = (c_1, \ldots, c_{2p})$ be a bipartite graceful permutation with $c_1 = y$.  This exists by Lemma~\ref{lem:bgp}.  Also by that lemma, we know that $c_{2p} = y + p$.

Without loss of generality, assume that $a_i = x$ and $a_{i+1} = y$ (if~$x$ and~$y$ appear in the other order, replace~${\bf a}$ with its reverse and then reverse the resulting graceful permutation once the construction is complete).  Consider the sequence
$$( a_1 + p, \ldots, a_i + p, c_1, c_2+q, \ldots, c_{2p-1}, c_{2p} + q, a_{i+1}+p, \ldots, a_q + p).$$
That is, the first~$i$ elements of ${\bf a}$ with~$p$ added to each, followed by~${\bf c}$ with~$q$ added to each element in an even position within~${\bf c}$, followed by the last~$q-i$ elements of ${\bf a}$ with~$p$ added to each.   We have
$$\{ c_1, c_3, \ldots, c_{2p-1} \} = [0, p-1], \hspace{3mm} \{ a_1 + p, \ldots, a_q + p \} = [p, p+q-1]$$
and
$$\{ c_2+q, c_4+q, \ldots, c_{2p}+q \} = [p+q, p+q-1],$$
so the sequence has the required elements.

As ${\bf a}$ is a graceful permutation, we have the absolute differences $[1, q-1] \setminus \{y-x\}$ and as ${\bf c}$ is a bipartite graceful permutation, we have the absolute differences $[ q+1, 2p+q-1 ]$.  Two additional absolute differences are given at the joins:
$$ |c_1 - (a_i+p)| =  (x +2(y-x)) - y = y-x$$
and
$$ | (a_{i+1}+p) -  (c_{2p} + q)| = (y + p+q) - (y+p) = q.$$
Hence our sequence is a graceful permutation. Its first element is~$a_1+p$ and its last element is $a_q+p$.
\end{proof}

Call the graceful permutation obtained from a graceful permutation~${\bf a}$ and a bipartite graceful permutation~${\bf c}$ via the method of Theorem~\ref{th:insert}'s proof the {\em insertion} of  ${\bf c}$ into~${\bf a}$ at $i$. 

\begin{exa}\label{ex:insert}
The insertion of 
$${\bf c} = (4,11,5,10,6,9,7,8,0,15,1,14,2,13,3,12)$$
into ${\bf a} = (1,6,0,4,3,5,2)$ at~$i = 3$ (where $x=0$, $y=4$, $2(y-x) = p = 8$ and $q=7$)  is
$$(9,14,8, 4,18,5,17,6,16,7,15,0,22,1,21,2,20,3,19, 12,11,13,10),$$
a graceful permutation of length~$23$.
\end{exa}

We can often  guarantee options to choose as~$x$ and~$y$ in Theorem~\ref{th:insert}:

\begin{lem}\label{lem:insertg}
Let~${\bf a}$ be a graceful permutation of length~$q$.  There is a valid insertion point for a bipartite graceful permutation of length~$2p$ whenever $p$ is even and $2(q-1)/3 < p < 2q$.
\end{lem}

\begin{proof}
The conditions that $p$ be even and~$p < 2q$ guarantee that $p/2$ is an absolute difference of~${\bf a}$.  Let~$p/2 = y-x$ for some adjacent elements~$x,y$ of~${\bf a}$.  

If~$q \leq p$ then $2(y-x) \geq q > y$ and the conditions of Theorem~\ref{th:insert} are met.  If $q>p$ then in the complement of~${\bf g}$ there are adjacent elements~$y' = q-1-x$ and~$x' = q-1-y$ with~$y'-x' = p/2$ and $2(y'-x') > y' >x'$.  Apply Theorem~\ref{th:insert} and take the complement of the result.
\end{proof}

Combining the insertion method with Lemma~\ref{lem:g2s} we can prove the following result about the existence of $S$-sequencings.

\begin{thm}
Let~$\ell$ be an odd multiple of~$3$.
Let $S \subseteq D_{8\ell + 6} \setminus \{ e \}$ with $|S| = 8\ell + 4$ and $x$ the non-identity element not in~$S$.  If~$x$ has order $(4\ell + 3)/3$ then $D_{8\ell+6}$ has an $S$-sequencing.
\end{thm}

\begin{proof}
We aim to use the third clause of Lemma~\ref{lem:g2s} and hence are looking to construct graceful permutations of length~$2\ell+1$ with first element~$\ell$ and final element~$\ell - 1$.    In all cases the first absolute difference of this graceful permutation will be~3.   For a given target length, we construct the permutation with applications of the insertion construction (usually repeatedly).

Given a graceful permutation of odd length~$q$ with first element~$(q-1)/2$, first absolute difference~3 and final element $(q-3)/2$ (that is, of the form we are searching for) we can use the insertion method  and Lemma~\ref{lem:insertg} to find a graceful permutation of length~$2\ell+1 = q + 2p$ of the form we require, provided that $2(q-1)/3 < p < 2q$ and $p/2 \neq 3$.

Therefore, in general, success at length~$2\ell+1$ follows from success at length~$q$ whenever $(2\ell+1)/5 < q < (6\ell+7)/7$.   There is one exception: we cannot go from~$q=7$ to~$2\ell+1 = 19$ as this requires $p/2 = 3$.  Hence it is sufficient to find successful permutations for $q \in \{ 7, 11, 15, 19 \}$.  

The construction in Example~\ref{ex:gp1} covers the case~$q=7$.  Here are graceful permutations for the remaining cases:
$$\begin{array}{c}
(5, 8, 0, 10, 1, 6, 7, 3, 9, 2, 4), \\
(7, 10, 0, 14, 1, 13, 2, 8, 3, 12, 4, 11, 9, 5, 6), \\
(9, 12, 0, 18, 1, 17, 2, 16, 3, 7, 13, 5, 14, 4, 15, 8, 6, 11, 10). \\
\end{array}$$
\end{proof}

More generally, the method of proof here can be applied for any fixed element in place of~3.

\begin{thm}\label{th:fixd}
Fix~$d$.  Let $S \subseteq D_{2m} \setminus \{ e \}$ with $|S| = 2m-2$ and $x$ the non-identity element not in~$S$.  If $m \equiv 5,7 \pmod{8}$ and~$x$ has order $m/d$, then $D_{2m}$ has an $S$-sequencing, except possibly in finitely many cases.
\end{thm}

\begin{proof}
Example~\ref{ex:gp1} gives graceful permutations of lengths $2\ell$ and $2\ell + 1$ that can be used in  the second and third clauses of Lemma~\ref{lem:g2s} to give an $S$-sequencing for $S = D_{2m}  \setminus \{ e, u^d \}$ for two values of~$m$, one of the form~$4\ell+1$ and one of the form~$4\ell +3$ (odd $\ell$). 

We can then repeatedly use Lemma~\ref{lem:insertg} to construct lots of additional graceful permutations that imply success for $S = D_{2m'}  \setminus \{ e, u^d \}$ for all but finitely many values of~$m'$.
\end{proof}

For cracked graceful permutations the same proof as Gv\"ozdjak's for part~1 of  Theorem~\ref{th:gvoz} applies:

\begin{lem}
Let~$0 \leq x,y \leq n$ and~$x \neq y$.  If there is a cracked graceful permutation of length~$n$ with first element~$x$ and last element~$y$ then $|x-y|$ has the same parity as $\lfloor n/2 \rfloor$.
\end{lem}

\begin{proof}
Let~$(a_1, a_2, \ldots, a_n)$, where~$a_1=x$ and~$a_n=y$, be a cracked graceful permutation.  Then
\begin{eqnarray*}
|x-y| &\equiv& |a_1 - a_2| + |a_2 - a_3| + \cdots + |a_{n-1} - a_n|  \\
       & \equiv &  n(n-1)/2  \ \equiv \lfloor n/2 \rfloor \pmod{2}
\end{eqnarray*}
as required.
\end{proof}

This result, along with Theorem~\ref{th:gvoz}, implies that neither Isbell's second nor third constructions can successfully cover all dihedral groups of the form~$D_{8\ell + 6}$ alone.

It is possible to extend the insertion method to cracked graceful permutations and then employ a similar approach to that of Theorem~\ref{th:fixd}.  However, we are lacking a graceful permutation that would get us started in the case~$m \equiv 3 \pmod{8}$.  Rather than pursue this route here, we use the specific construction of the cracked graceful permutation in Example~\ref{ex:hgp}, which allows a more efficient way to cover some small initial elements of sequencings of~$D_{8\ell+6}$.

\begin{thm}\label{lem:isb3start}
For $d \in \{5,6,7\}$ and even~$\ell > 2$, there is a cracked graceful permutation of length~$2\ell - 1$ with crack~$\ell -2$ that starts with~$\ell + 1$, ends with~$\ell$ and has first absolute difference~$d$, except when~$\ell=4$ and~$d \in \{6,7\}$.
\end{thm}

\begin{proof}
We prove the bulk of this result by varying the start of the Isbell construction given in Example~\ref{ex:hgp}.

When~$\ell \equiv 0 \pmod{6}$, for~$t < \ell/6$ the first $6t + 3$ elements of Isbell's cracked graceful permutation are a translate of an arrangement of the elements~$[-3t-2, 3t+2]  \setminus \{ -1, 1 \}$, starting with~2 and ending with~$-3t$, and having absolute differences~$[2, 6t+4]  \setminus \{3\}  $.

When~$\ell \equiv 2 \pmod{6}$ with~$\ell \geq 14$, for~$t < (\ell-8)/6$  the first $6t + 11$ elements of Isbell's cracked graceful permutation are a translate of an arrangement of the elements~$[-3t-6, 3t+6]  \setminus \{ -1, 1 \}$, starting with~2 and ending with~$-3t-4$, and having absolute differences~$[2, 6t+12] \setminus \{3\}   $.

When~$\ell \equiv 4 \pmod{6}$ with~$\ell \geq 10$, for~$t < (\ell-4)/6$  the first $6t + 7$ elements of Isbell's cracked graceful permutation are a translate of an arrangement of the elements~$[-3t-4, 3t+4]  \setminus \{ -1, 1 \}$, starting with~2 and ending with~$-3t-2$, and having absolute differences~$[2, 6t+8]  \setminus \{3\}  $.

In each case we may substitute an alternative sequence with these properties to obtain an alternative cracked graceful permutation.   Table~\ref{tab:crackedstart} gives the sequences required to prove the result, except when~$(l,d)$ is one of
$$ (4,5), (6,5), (6,6), (6,7), (8,5),(8,6),(8,7), (10,6),(10,7), (12,5), (14,6), (16,7)    .$$ 
Cracked graceful permutations with the stated properties for these parameters are given in Table~\ref{tab:crackedend}.

\begin{table}
\caption{Sequences for use in the proof of Lemma~\ref{lem:isb3start}}\label{tab:crackedstart}
$$
\begin{array}{rrrl}
\hline
\ell \hspace{-3mm}  \pmod{6} & d &  t & {\rm sequence}  \\
\hline
0 & 5 & 2 & (2, -3, 3, -4, 7, -7, 8, -8, 5, -5, 4, 0, -2, 6, -6) \\
   & 6 & 1 & (2, -4, 5, -5, 3, -2, 0, 4, -3) \\
   & 7 & 1 & (2, -5, 5, -4, 4, 0, -2, 3, -3) \\
2 & 5 & 0 & (2, -3, 3, -6, 6, -5, 5, -2, 0, 4, -4) \\
   & 6 & 1 & (2, -4, 3, -2, 0, 4, -6, 7, -8, 8, -9, 9, -5, 6, -3, 5, -7) \\
   & 7 & 0 & (2, -5, 5, -6, 6, -3, 3, -2, 0, 4, -4) \\
4 & 5 & 0 & (2, -3, 3, -4, 4, 0, -2) \\
   & 6 & 1 & (2, -4, 3, -2, 0, 4, -7, 7, -6, 6, -3, 5, -5) \\
   & 7 & 2 & (2, -5, 4, 0, -2, 3, -3, 5, -7, 6, -4, 7, -10, 10, -9, 9, -6, 8, -8 ) \\
\hline
\end{array}
$$
\end{table}

\begin{table}
\caption{Some cracked graceful permutations}\label{tab:crackedend}
$$
\begin{array}{rrl}
\hline
\ell  & d  & {\rm cracked \ graceful \ permutation}  \\
\hline
5 & 5 & (5, 0, 1, 7, 3, 6, 4) \\
6 & 5 & (7, 2, 5, 11, 3, 1, 8, 9, 0, 10, 6) \\
   & 6 & (7, 1, 5, 10, 0, 3, 11, 2, 9, 8, 6) \\
   & 7 & (7, 0, 2, 10, 1, 11, 5, 9, 8, 3, 6) \\
8 & 5 & (9, 4, 12, 3, 13, 14, 2, 15, 1, 5, 11, 0, 7, 10, 8) \\
   & 6 & (9, 3, 10, 14, 2, 1, 15, 5, 7, 12, 4, 13, 0, 11, 8) \\
   & 7 & (9, 2, 13, 11, 7, 1, 14, 0, 10, 15, 3, 12, 4, 5, 8) \\
10 & 6 & (11, 5, 13, 15, 0, 1, 19, 2, 18, 4, 16, 3, 12, 9, 14, 7, 17, 6, 10) \\
     & 7 & (11, 4, 13, 15, 0, 5, 6, 12, 9, 17, 7, 18, 1, 19, 3, 16, 2, 14, 10) \\
12 & 5 &  (13, 8, 15, 17, 0, 23, 1, 2, 22, 3, 21, 5, 19, 4, 16, 6, 14, 11, 7, 20, 9, 18, 12) \\
 14 & 6 & (15, 9, 17, 19, 1, 27, 2, 26, 3, 25, 4, 0, 20,   \\
      &     & \hspace{5mm} 10, 21, 7, 16, 23, 8, 11, 24, 5, 22, 6, 18, 13, 14) \\  
  16 & 7 & (17, 10, 19, 21, 1, 31, 2, 30, 3, 29, 4, 28, 5, 27, 6, \\
       &    & \hspace{5mm}  0, 18, 15, 11, 25,  20, 8, 23, 7, 26, 9, 22, 12, 13, 24, 16) \\
\hline
\end{array}
$$
\end{table}

\end{proof}

This immediately implies:

\begin{cor}\label{cor:cracked}
Let~$\ell$ be even.
Let $S \subseteq D_{8\ell + 6} \setminus \{ e \}$ with $|S| = 8\ell + 4$ and $x$ the non-identity element not in~$S$.  If~$x$ has order $(4\ell + 3)/d$ for $d \in \{ 3,5,7 \}$ then $D_{8\ell+6}$ has an $S$-sequencing.
\end{cor}

\section{Strong Sequenceability}\label{sec:ss}

As usual, let ${\bf g} = (g_1, g_2, \ldots, g_k)$ be an arrangement of elements of~$G \setminus \{e \}$ with partial product sequence~${\bf h} = (h_0, h_1, h_2, \ldots, h_k)$.  Suppose the elements of~${\bf g}$ are distinct and let~$S = \{ g_1, g_2, \ldots, g_k \}$.

If the  elements $(h_1, \ldots h_k)$ are all distinct with~$h_k = h_0 = e$ then~${\bf h}$ is a {\em rotational directed $S$-terrace} for~$G$ and~${\bf g}$ is the associated~{\em  rotational $S$-sequencing} of~$G$.  If $|G |= n$ and~$k = n-1$ we get a rotational directed terrace and associated rotational sequencing of~$G$ as in Section~\ref{sec:largek}.  

A group is {\em strongly sequenceable} if for all $S \subseteq G \setminus \{e\}$ there is either an $S$-sequencing or a rotational $S$-sequencing (or both).  In an abelian group the element~$h_k$ is independent of the ordering and so we cannot have both an $S$-sequencing and a rotational $S$-sequencing.  However, in non-abelian groups it is possible to have both.  For example, the dihedral groups of order a multiple of~4 have both sequencings and rotational sequencings~\cite{Isbell90,Keedwell83,Li97}.  Alspach and Kalinowski, see~\cite{AL}, have posed the problem of determining which groups are strongly sequenceable.

As observed in~\cite{AL}, $D_6$ has neither a sequencing nor a rotational sequencing.   The same is true for the quaternion group~$Q_8$ \cite{FGM, Gordon61}.    Hence we have:

\begin{thm}\label{th:ss}
If $G$ has a subgroup isomorphic to either~$D_6$ or~$Q_8$, then~$G$ is not strongly sequenceable.
\end{thm}

The other sets~$S$ that do not have $S$-sequencings discussed in Section~\ref{sec:smallk} either have rotational $S$-sequencings or do not imply that any further groups are not strongly sequenceable.

It is known that all abelian groups of order at most~23 and all cyclic groups of order at most~25 are strongly sequenceable~\cite{ADMS,CMPP18}.  By the remark at the end of Section~\ref{sec:polymeth} we can add~$D_{10}$ to this list.

\section*{Acknowledgements}

I am very grateful to John Schmitt (Middlebury College) for conversations about the polynomial method, his close attention to sections of this work and his encouragement.  I am also grateful to Kat Cannon-MacMartin (Marlboro College) for the programming used in the  proof of Theorem~\ref{th:polymeth} and to Brian Alspach (University of Newcastle) for sharing work on strong sequenceability.

\end{document}